\documentclass[12pt,oneside]{article}

\usepackage{setspace}

\usepackage{amssymb}   

\usepackage{amsmath}

\usepackage{mathrsfs}

\usepackage{stmaryrd}

\usepackage{amsthm}
\usepackage{newlfont}
\usepackage{amscd}
\usepackage{mathtools}
\usepackage{bm}
\usepackage{tikz-cd}
\usepackage{graphicx}

\usepackage{hyperref}

\usepackage{enumitem}

\usepackage[all]{xy}

\usepackage{textcomp}

\usepackage{amsbsy}

\newtheorem{thm}{Theorem}[section]

\newtheorem{thm-defn}[thm]{Theorem/Definition}

\newtheorem{prop}[thm]{Proposition}

\theoremstyle{definition}

\newtheorem{ques}[thm]{Question}

\theoremstyle{remark}
\newtheorem{rem}[thm]{Remark}

\numberwithin{equation}{section}

\begin{document}

\pagenumbering{arabic}

\title{Relative crystalline representations and weakly admissible modules}
\author{Yong Suk Moon}
\date{}

\maketitle

\begin{abstract}
Let $k$ be a perfect field of characteristic $p > 2$, and let $K$ be a finite totally ramified extension over $W(k)[\frac{1}{p}]$. Let $R_0$ be an unramified relative base ring over $W(k)\langle X_1^{\pm 1}, \ldots, X_d^{\pm 1}\rangle$, and let $R = R_0\otimes_{W(k)}\mathcal{O}_K$. We define relative $B$-pairs and study their relations to weakly admissible $R_0[\frac{1}{p}]$-modules and $\mathbf{Q}_p$-representations. As an application, when $R = \mathcal{O}_K[\![Y]\!]$ with $k = \overline{k}$, we show that every rank $2$ horizontal crystalline representation with Hodge-Tate weights in $[0, 1]$ whose associated isocrystal over $W(k)[\frac{1}{p}]$ is reducible arises from a $p$-divisible group over $R$. Furthermore, we give an example of a $B$-pair which arises from a weakly admissible $R_0[\frac{1}{p}]$-module but does not arise from a $\mathbf{Q}_p$-representation. 
\end{abstract}

\tableofcontents

\section{Introduction} \label{sec:1}

Let $k$ be a perfect field of characteristic $p > 2$, and let $W(k)$ be its ring of Witt vectors. Let $K$ be a finite totally ramified extension over $W(k)[\frac{1}{p}]$, and denote by $\mathcal{O}_K$ its ring of integers. Let $R_0$ be an unramified relative base ring over $W(k)\langle X_1^{\pm 1}, \ldots, X_d^{\pm 1}\rangle$ which is the $p$-adic completion of $W(k)[X_1^{\pm 1}, \ldots, X_d^{\pm 1}]$, and let $R = R_0\otimes_{W(k)}\mathcal{O}_K$. (cf. Section \ref{sec:2.1}). Examples of such $R$ include $\mathcal{O}_K\langle X_1^{\pm 1}, \ldots, X_d^{\pm 1}\rangle$ and the formal power series ring $\mathcal{O}_K[\![Y_1, \ldots, Y_d]\!]$ with $Y_i = X_i-1$.

Brinon developed $p$-adic Hodge theory in the relative case in \cite{brinon-relative}, which is studied further by Scholze in \cite{scholze-p-adic-hodge} and Kedlaya-Liu in \cite{kedlaya-liu-relative-padichodge}. Let $\overline{R}$ denote the union of finite $R$-subalgebras $R'$ of a fixed separable closure of $\mathrm{Frac}(R)$ such that $R'[\frac{1}{p}]$ is \'{e}tale over $R[\frac{1}{p}]$. Then $\mathrm{Spec}\overline{R}[\frac{1}{p}]$ is a pro-universal covering of $\mathrm{Spec}R[\frac{1}{p}]$, and $\overline{R}$ is the integral closure of $R$ in $\overline{R}[\frac{1}{p}]$. Let $\mathcal{G}_R \coloneqq \mathrm{Gal}(\overline{R}[\frac{1}{p}]/R[\frac{1}{p}]) = \pi_1^{\text{\'{e}t}}(\mathrm{Spec}R[\frac{1}{p}])$. In \cite{brinon-relative}, the relative crystalline period ring $B_{\mathrm{cris}}(R)$ is constructed, and the notions of \emph{crystalline} representations of $\mathcal{G}_R$ and \emph{filtered} $(\varphi, \nabla)$-\emph{modules over} $R_0[\frac{1}{p}]$ are defined generalizing those when the base is a $p$-adic field. In \emph{loc. cit.}, \emph{punctually weakly admissible} modules and \emph{weakly admissible} modules are also defined to generalize weakly admissible modules over a $p$-adic field. A fundemental open question concerning these objects is the following:

\begin{ques} \label{ques:1.1}
Which filtered $(\varphi, \nabla)$-modules over $R_0[\frac{1}{p}]$ arise from crystalline representations of $\mathcal{G}_R$?	
\end{ques}

A filtered $(\varphi, \nabla)$-module over $R_0[\frac{1}{p}]$ is said to be \textit{admissible} if it arises from a crystalline representation of $\mathcal{G}_R$. When the base is a $p$-adic field $K$, it is proved in \cite{colmez-fontaine} that a filtered $\varphi$-module over $K$ with zero monodromy is admissible if and only if it is weakly admissible. 

Another interesting question in relative $p$-adic Hodge theory concerns representations arising from a $p$-divisible group. For a $p$-divisible group $G_R$ over $\mathrm{Spec} R$, let $T_p(G_R) \coloneqq \mathrm{Hom}_{\overline{R}}(\mathbf{Q}_p/\mathbf{Z}_p, ~G_R\times_R \overline{R})$ be the associated Tate module. Then by \cite[Corollary 5.4.2]{kim-groupscheme-relative}, $(T_p(G_R)\otimes_{\mathbf{Z}_p}\mathbf{Q}_p)^{\vee}$ is a crystalline $\mathcal{G}_R$-representation whose Hodge-Tate weights lie in $[0, 1]$ (in this paper, we use the covariant version of the functor $D_{\mathrm{dR}}(\cdot)$ to define Hodge-Tate weights, which is different from the convention using the contravariant one). This raises the following natural question.

\begin{ques} \label{ques:1.2}
	Which crystalline representations of $\mathcal{G}_R$ whose Hodge-Tate weights lie in $[0, 1]$ arise from $p$-divisible groups over $\mathrm{Spec} R$?
\end{ques}

Kim showed in \cite[Theorem 3.5]{kim-groupscheme-relative} that the category of $p$-divisible groups over $R$ is anti-equivalent to the category of relative Breuil modules, which characterize the linear algebraic structure of corresponding weakly admissible modules. Hence, for crystalline representations with Hodge-Tate weights in $[0, 1]$, Question \ref{ques:1.2} is closely related to Question \ref{ques:1.1}. When the base is a $p$-adic field $K$, Kisin proved in \cite[Corollary 2.2.6]{kisin-crystalline} that every crystalline $\mathrm{Gal(\overline{K}/K)}$-representation whose Hodge-Tate weights lie in $[0, 1]$ arises from a $p$-divisible group over $\mathcal{O}_K$.  

In this paper, our objects of study center around Question \ref{ques:1.1} and \ref{ques:1.2} for some special cases. We define the category of $B$-pairs in the relative case and study the relations with $\mathbf{Q}_p$-representations and weakly admissible modules. As an application, when $R = \mathcal{O}_K[\![Y]\!]$ with $k = \overline{k}$, we compute $B$-pairs corresponding to certain weakly admissible modules and show the following theorem.

\begin{thm} \label{thm:1.3}
Let $R = \mathcal{O}_K[\![Y]\!]$ and suppose $k$ is algebraically closed. Let $V$ be a horizontal crystalline $\mathcal{G}_R$-representation of rank $2$ over $\mathbf{Q}_p$ with Hodge-Tate weights in $[0, 1]$ such that its associated isocrystal is reducible. Then there exists a $p$-divisible group $G_R$ over $R$ such that $(T_p(G_R)[\frac{1}{p}])^{\vee} \cong V$ as $\mathcal{G}_R$-representations. Furthermore, there exists a $B$-pair which arises from a weakly admissible $R_0[\frac{1}{p}]$-module but does not arise from a $\mathbf{Q}_p$-representation. 
\end{thm}

In particular, the last statement of Theorem \ref{thm:1.3} shows that the relative case is different from the case when the base is a $p$-adic field where every semi-stable $B$-pair of slope $0$ arises from a $\mathbf{Q}_p$-representation. It also answers negatively the question raised in \cite{brinon-relative} whether weakly admissible implies admissible in the relative case.

We remark that it is proved in \cite{liu-moon-rel-cryst} using a completely different method that when $R$ has Krull dimension $2$ with the ramification index $e < p-1$, then every crystalline representation with Hodge-Tate weights in $[0, 1]$ comes from a $p$-divisible group over $R$. However, the argument in \cite{liu-moon-rel-cryst} relies crucially on the assumption that the ramification is small, whereas the result in Theorem \ref{thm:1.3} holds for any ramification.

\section*{Acknowledgement}
I would like to express my sincere gratitude to Tong Liu for many helpful discussions and suggestions on this topic.

\section{$p$-adic Hodge Theory in the relative case} \label{sec:2}

\subsection{Crystalline and de Rham period rings} \label{sec:2.1}

We follow the notation as in the Introduction. We first recall the constructions and results of relative $p$-adic Hodge theory developed in \cite{brinon-relative}, using the same terminologies such as \emph{punctually weakly admissible modules} and \emph{weakly admissible modules}. Denote by $W(k)\langle X_1^{\pm 1}, \ldots, X_d^{\pm 1}\rangle$ the $p$-adic completion of the polynomial ring $W(k)[X_1^{\pm 1}, \ldots, X_d^{\pm 1}]$. Let $R_0$ be a ring obtained from $W(k)\langle X_1^{\pm 1}, \ldots, X_d^{\pm 1}\rangle$ by a finite number of iterations of the following operations:

\begin{itemize}
\item $p$-adic completion of an \'{e}tale extension;
\item $p$-adic completion of a localization;
\item completion with respect to an ideal containing $p$.	
\end{itemize}

\noindent We further assume that either $W(k)\langle X_1^{\pm 1}, \ldots, X_d^{\pm 1}\rangle \rightarrow R_0$ has geometrically regular fibers or $R_0$ has Krull dimension less than $2$, and that $k \rightarrow R_0/pR_0$ is geometrically integral and $R_0$ is an integral domain. 

$R_0/pR_0$ has a finite $p$-basis given by $X_1, \ldots, X_d$. The Witt vector Frobenius on $W(k)$ extends (not necessarily uniquely) to $R_0$, and we fix such a Frobenius endomorphism $\varphi: R_0 \rightarrow R_0$. Let $\widehat{\Omega}_{R_0} = \varprojlim_{n} \Omega_{(R_0/p^n)/W(k)}$ be the module of $p$-adically continuous K\"{a}hler differentials. Then $\widehat{\Omega}_{R_0} \cong \bigoplus_{i=1}^d R_0 \cdot d{X_i}$ by \cite[Proposition 2.0.2]{brinon-relative}. If $\nabla: R_0[\frac{1}{p}] \rightarrow R_0[\frac{1}{p}]\otimes_{R_0} \widehat{\Omega}_{R_0}$ is the universal continuous derivation, then $(R_0[\frac{1}{p}])^{\nabla = 0} = W(k)[\frac{1}{p}]$. We work over the base ring $R$ given by  $R \coloneqq R_0\otimes_{W(k)}\mathcal{O}_K$.

The relative de Rham period ring and the crystalline period ring are constructed as follows. Let $\displaystyle \overline{R}^{\flat} = \varprojlim_{\varphi} \overline{R}/p\overline{R}$. There exists a natural $W(k)$-linear surjective map $\theta: W(\overline{R}^{\flat}) \rightarrow \widehat{\overline{R}}$ which lifts the projection onto the first factor. Here, $\widehat{\overline{R}}$ denotes the $p$-adic completion of $\overline{R}$. Define $B_{\mathrm{dR}}^{\nabla +}(R) \coloneqq \varprojlim_{n} W(\overline{R}^{\flat})[\frac{1}{p}]/(\mathrm{ker}(\theta))^n$. Choose compatibly $\epsilon_n \in \overline{R}$ such that $\epsilon_0 =1, ~\epsilon_n = \epsilon_{n+1}^p$ with $\epsilon_1 \neq 1$, and let $\widetilde{\epsilon} = (\epsilon_n)_{n \geq 0} \in \overline{R}^{\flat}$. Then $t \coloneqq \log{[\widetilde{\epsilon}]} \in B_{\mathrm{dR}}^{\nabla +}(R)$ and $B_{\mathrm{dR}}^{\nabla +}(R)$ is $t$-torsion free. The horizontal de Rham period ring is defined to be $\displaystyle B_{\mathrm{dR}}^{\nabla}(R) = B_{\mathrm{dR}}^{\nabla +}(R)[\frac{1}{t}]$, equipped with the filtration $\mathrm{Fil}^j B_{\mathrm{dR}}^{\nabla}(R) = t^j B_{\mathrm{dR}}^{\nabla +}(R)$ for $j \in \mathbf{Z}$. The $\mathcal{G}_R$-action on $W(\overline{R}^{\flat})$ extends uniquely to $B_{\mathrm{dR}}^{\nabla}(R)$. Let $\theta_R: R\otimes_{W(k)}W(\overline{R}^{\flat}) \rightarrow \widehat{\overline{R}}$ be the $R$-linear extension of $\theta$, and denote by $A_{\mathrm{inf}}(\widehat{\overline{R}}/R)$ the completion of $R\otimes_{W(k)}W(\overline{R}^{\flat})$ for the topology given by the ideal $\theta_R^{-1}(p\widehat{\overline{R}})$. Let $B_{\mathrm{dR}}^+(R) = \varprojlim_n A_{\mathrm{inf}}(\widehat{\overline{R}}/R)[\frac{1}{p}]/(\mathrm{ker}(\theta_R))^n$. Define the de Rham period ring to be $B_{\mathrm{dR}}(R) = B_{\mathrm{dR}}^+(R)[\frac{1}{t}]$. For $j \geq 0$, we let $\mathrm{Fil}^j B_{\mathrm{dR}}^+(R) = (\mathrm{ker}(\theta_R))^j$ and $\displaystyle\mathrm{Fil}^0 B_{\mathrm{dR}}(R) = \sum_{n = 0}^{\infty} \frac{1}{t^n}\mathrm{Fil}^n B_{\mathrm{dR}}^+(R)$. For $j \in \mathbf{Z}$, let $\mathrm{Fil}^j B_{\mathrm{dR}}(R) = t^j \mathrm{Fil}^0 B_{\mathrm{dR}}(R)$. $B_{\mathrm{dR}}(R)$ is equipped with the connection $\nabla: B_{\mathrm{dR}}(R) \rightarrow B_{\mathrm{dR}}(R)\otimes_{R_0} \widehat{\Omega}_{R_0}$ which is $W(\overline{R}^{\flat})$-linear and extends the universal continuous derivation of $R$. $\nabla$ satisfies the Griffiths transversality. The $\mathcal{G}_R$-action on $R\otimes_{W(k)}W(\overline{R}^{\flat})$ extends uniquely to $B_{\mathrm{dR}}(R)$, and commutes with $\nabla$. We have a natural embedding $B_{\mathrm{dR}}^{\nabla}(R) \hookrightarrow B_{\mathrm{dR}}(R)$ compatible with the filtrations and $\mathcal{G}_R$-actions, and $B_{\mathrm{dR}}^{\nabla}(R) = (B_{\mathrm{dR}}(R))^{\nabla = 0}$. Furthermore, $B_{\mathrm{dR}}(R)^{\mathcal{G}_R} = R[\frac{1}{p}]$ and $(B_{\mathrm{dR}}^{\nabla}(R))^{\mathcal{G}_R} = K$.

For $i = 1, \ldots, d$, choose compatibly $X_{i, n} \in \overline{R}$ such that $X_{i, 0} = X_i$ and $X_{i, n} = X_{i, n+1}^p$, and let $[\widetilde{X_i}] \in \overline{R}^{\flat}$. Let $u_i = X_i\otimes 1-1\otimes [\widetilde{X_i}] \in R\otimes_{W(k)} W(\overline{R}^{\flat})$. The following proposition is proved in \cite{brinon-relative}.

\begin{prop} \label{prop:2.1} \emph{(cf. \cite[Proposition 5.1.4, 5.2.2, 5.2.5]{brinon-relative})}
The natural embedding 
\[
B_{\mathrm{dR}}^{\nabla +}(R)[\![u_1, \ldots, u_d]\!] \rightarrow B_{\mathrm{dR}}^{+}(R)
\]
is an isomorphism. Furthermore, $\displaystyle \mathrm{Fil}^0 B_{\mathrm{dR}}(R) = B_{\mathrm{dR}}^+(R)[\frac{u_1}{t}, \ldots, \frac{u_d}{t}]$.
\end{prop}
  
For the horizontal crystalline and crystalline period rings, we first construct the integral ones. Let $A_{\mathrm{cris}}^{\nabla}(R)$ be the $p$-adic completion of the divided power envelope of $W(\overline{R}^{\flat})$ with respect to $\mathrm{ker}(\theta)$. The Witt vector Frobenius and $\mathcal{G}_R$-action on $W(\overline{R}^{\flat})$ extend uniquely to $A_{\mathrm{cris}}^{\nabla}(R)$. Let $\theta_{R_0}: R_0\otimes_{W(k)}W(\overline{R}^{\flat}) \rightarrow \widehat{\overline{R}}$ be the $R_0$-linear extension of $\theta$, and define $A_{\mathrm{cris}}(R)$ to be the $p$-adic completion of the divided power envelope of $R_0\otimes_{W(k)}W(\overline{R}^{\flat})$ with respect to $\mathrm{ker}(\theta_{R_0})$. The $\mathcal{G}_R$-action on $R_0\otimes_{W(k)}W(\overline{R}^{\flat})$ extends uniquely to $A_{\mathrm{cris}}(R)$. The Frobenius endomorphism on $R_0\otimes_{W(k)}W(\overline{R}^{\flat})$ given by $\varphi$ on $R_0$ and the Witt vector Frobenius on $W(\overline{R}^{\flat})$ extends uniquely to $A_{\mathrm{cris}}(R)$. We have the connection $\nabla: A_{\mathrm{cris}}(R) \rightarrow A_{\mathrm{cris}}(R)\otimes_{R_0} \widehat{\Omega}_{R_0}$ which is $W(\overline{R}^\flat)$-linear and extends the universal continuous derivation of $R_0$. The Frobenius on $A_{\mathrm{cris}}(R)$ is horizontal. We have a natural $\mathcal{G}_R$-equivariant embedding $A_{\mathrm{cris}}^{\nabla}(R) \hookrightarrow A_{\mathrm{cris}}(R)$, and $A_{\mathrm{cris}}(R)^{\nabla = 0} = A_{\mathrm{cris}}^{\nabla}(R)$. Moreover, $(A_{\mathrm{cris}}^{\nabla}(R))^{\mathcal{G}_R} = W(k)$ and $A_{\mathrm{cris}}(R)^{\mathcal{G}_R} = R_0$. Note that $t \in A_{\mathrm{cris}}(\mathbf{Z}_p)$ and $p$ divides $t^{p-1}$ in $A_{\mathrm{cris}}(\mathbf{Z}_p)$. $A_{\mathrm{cris}}(R)$ is $t$-torsion free, and we define $B_{\mathrm{cris}}^{\nabla}(R) = A_{\mathrm{cris}}^{\nabla}(R)[\frac{1}{t}]$ and $B_{\mathrm{cris}}(R) = A_{\mathrm{cris}}(R)[\frac{1}{t}]$, equipped with the Frobenius and $\mathcal{G}_R$-action extending those on $A_{\mathrm{cris}}(R)$. We extend the connection on $A_{\mathrm{cris}}(R)$ to $B_{\mathrm{cris}}(R)$ $t$-linearly. $B_{\mathrm{cris}}(R)\otimes_{R_0[\frac{1}{p}]} R[\frac{1}{p}]$ naturally embeds into $B_{\mathrm{dR}}(R)$ compatibly with the connections and $\mathcal{G}_R$-actions. 

Let $U_1 = \{x \in B_{\mathrm{cris}}^{\nabla}(R) \cap B_{\mathrm{dR}}^{\nabla +}(R), ~\varphi(x) = px\}$. The following proposition is shown in \cite{brinon-relative}.

\begin{prop} \label{prop:2.2} \emph{(cf. \cite[Lemma 6.2.22, Proposition 6.2.23]{brinon-relative})} The following sequences are exact:
\[
0 \rightarrow \mathbf{Q}_p\cdot t \rightarrow U_1 \stackrel{\theta}{\rightarrow} B_{\mathrm{dR}}^{\nabla +}(R)/\mathrm{Fil}^1 B_{\mathrm{dR}}^{\nabla +}(R) \rightarrow 0,
\]	
\[
0 \rightarrow \mathbf{Q}_p \rightarrow (B_{\mathrm{cris}}^{\nabla}(R))^{\varphi = 1} \rightarrow B_{\mathrm{dR}}^{\nabla}(R)/B_{\mathrm{dR}}^{\nabla +}(R) \rightarrow 0.
\]  
\end{prop}

For a continuous $\mathcal{G}_R$-representation $V$ over $\mathbf{Q}_p$, we denote $D_{\mathrm{cris}}^{\nabla}(V) = (V \otimes_{\mathbf{Q}_p}B_{\mathrm{cris}}^{\nabla}(R))^{\mathcal{G}_R}$, $D_{\mathrm{cris}}(V) = (V \otimes_{\mathbf{Q}_p}B_{\mathrm{cris}}(R))^{\mathcal{G}_R}$ and $D_{\mathrm{dR}}(V) = (V \otimes_{\mathbf{Q}_p}B_{\mathrm{dR}}(R))^{\mathcal{G}_R}$. The natural morphisms
\[
\begin{split}
&\alpha_{\mathrm{cris}}^{\nabla}: D_{\mathrm{cris}}^{\nabla}(V)\otimes_{W(k)[\frac{1}{p}]} B_{\mathrm{cris}}^{\nabla}(R) \rightarrow V\otimes_{\mathbf{Q}_p} B_{\mathrm{cris}}^{\nabla}(R),\\
&\alpha_{\mathrm{cris}}: D_{\mathrm{cris}}(V)\otimes_{R_0[\frac{1}{p}]} B_{\mathrm{cris}}(R)  \rightarrow V\otimes_{\mathbf{Q}_p}  B_{\mathrm{cris}}(R),\\
&\alpha_{\mathrm{dR}}: D_{\mathrm{dR}}(V)\otimes_{R[\frac{1}{p}]} B_{\mathrm{dR}}(R)  \rightarrow V\otimes_{\mathbf{Q}_p} B_{\mathrm{dR}}(R)
\end{split}
\] 
are injective. We say $V$ is \emph{horizontal crystalline} (resp. \emph{crystalline}, \emph{de Rham}) if $\alpha_{\mathrm{cris}}^{\nabla}$ (resp. $\alpha_{\mathrm{cris}}$, $\alpha_{\mathrm{dR}}$) is an isomorphism. For any $\mathbf{Q}_p$-representation $V$, we have natural embeddings $D_{\mathrm{cris}}^{\nabla}(V)\otimes_{W(k)[\frac{1}{p}]}R_0[\frac{1}{p}] \hookrightarrow D_{\mathrm{cris}}(V)$ and $D_{\mathrm{cris}}(V)\otimes_{R_0[\frac{1}{p}]} R[\frac{1}{p}] \hookrightarrow D_{\mathrm{dR}}(V)$. If $V$ is horizontal crystalline, then $V$ is crystalline and the map $D_{\mathrm{cris}}^{\nabla}(V)\otimes_{W(k)[\frac{1}{p}]}R_0[\frac{1}{p}] \rightarrow D_{\mathrm{cris}}(V)$ is an isomorphism. If $V$ is crystalline, then $V$ is de Rham and the map $D_{\mathrm{cris}}(V)\otimes_{R_0[\frac{1}{p}]} R[\frac{1}{p}] \rightarrow D_{\mathrm{dR}}(V)$ is an isomorphism. 

We study the linear algebraic structure of $D_{\mathrm{cris}}(V)$ in the following way. A \emph{filtered} $(\varphi, \nabla)$-\emph{module over} $R_0[\frac{1}{p}]$ is defined to be a tuple $(D, \varphi_D, \nabla_D, \mathrm{Fil}^j D)$ such that
\begin{itemize}
\item $D$ is a finite projective $R_0[\frac{1}{p}]$-module	;
\item $\varphi_D: D \rightarrow D$ is $\varphi$-semilinear endomorphism such that $1\otimes \varphi_D$ is an isomorphism;
\item $\nabla_D: D \rightarrow D\otimes_{R_0} \widehat{\Omega}_{R_0}$ is an integrable connection which is topologically quasi-nilpotent, i.e., there exists a finitely generated $R_0$-submodule $M \subset D$ stable under $\nabla$ such that $M[\frac{1}{p}] = D$ and the induced connection on $M/pM$ is nilpotent. The Frobenius $\varphi_D$ is horizontal with respect to $\nabla_D$;
\item $\mathrm{Fil}^j D_R$ is a decreasing separated and exhaustive filtration by $R[\frac{1}{p}]$-submodules of $D_R\coloneqq D\otimes_{R_0[\frac{1}{p}]}R[\frac{1}{p}]$ such that the graded module $\mathrm{gr}^{\bullet}D_R$ is projective over $R[\frac{1}{p}]$. Furthermore, Griffiths transversality holds for the induced connection: $\nabla_D(\mathrm{Fil}^j D_R) \subset \mathrm{Fil}^{j-1} D_R\otimes_{R_0} \widehat{\Omega}_{R_0}$.
\end{itemize}

Denote by $\mathrm{MF}(R)$ be the category of filtered $(\varphi, \nabla)$-modules over $R_0[\frac{1}{p}]$, whose morphisms are $R_0[\frac{1}{p}]$-module morphisms compatible with all structures. It is equipped with tensor product and duality structures as in \cite[Section 7]{brinon-relative}. For $D \in \mathrm{MF}(R)$, its Hodge-Tate weights are defined to be integers $w \in \mathbf{Z}$ such that $\mathrm{gr}^w D \neq 0$. We define its Hodge number 
\[
t_H(D) \coloneqq \sum_{j \in \mathbf{Z}} j \cdot \mathrm{rank}_{R[\frac{1}{p}]}(\mathrm{gr}^j D_R).
\]   
Let $\mathfrak{p} \in \mathrm{Spec} R_0/pR_0$, and let $\kappa_{\mathfrak{p}}$ be the perfect closure $\varinjlim_{\varphi} (R_0/pR_0)/\mathfrak{p}$. By the universal property of $p$-adic Witt vectors, there exists a unique map $b_{\mathfrak{p}}: R_0 \rightarrow W(\kappa_{\mathfrak{p}})$ lifting $R_0 \rightarrow \kappa_{\mathfrak{p}}$ which is compatible with Frobenius (with the Witt vector Frobenius on $W(\kappa_{\mathfrak{p}})$). Then $D_{\mathfrak{p}} \coloneqq D\otimes_{R_0, b_\mathfrak{p}} W(\kappa_{\mathfrak{p}})$ is a filtered $\varphi$-module over $W(\kappa_{\mathfrak{p}})[\frac{1}{p}]$ with the induced filtration and Frobenius. We define the \emph{Newton number of} $D$ \emph{at} $\mathfrak{p}$ to be the Newton number of $D_{\mathfrak{p}}$, and denote it by $t_N (D, \mathfrak{p})$. We say $D$ is \emph{punctually weakly admissible} if for all $\mathfrak{p} \in \mathrm{Spec}(R_0/pR_0)$, the following conditions hold:
\begin{itemize}
\item $t_H(D) = t_N(D, \mathfrak{p})$;
\item For each sub-object $D'$ of $D$ in $\mathrm{MF}(R)$, $t_H(D') \leq t_N(D', \mathfrak{p})$. 	
\end{itemize}
\noindent Denote by $\mathrm{MF}^{\mathrm{pwa}}(R)$ the full subcategory of $\mathrm{MF}(R)$ consisting of punctually weakly admissible modules. 

For a $\mathcal{G}_R$-representation $V$ over $\mathbf{Q}_p$, we equip $D_{\mathrm{cris}}^{\nabla}(V)$ with the Frobenius induced from $B_{\mathrm{cris}}(R)$, $D_{\mathrm{cris}}(V)$ with the Frobenius and connection induced from $B_{\mathrm{cris}}(R)$, and $D_{\mathrm{dR}}(V)$ with the filtration and connection induced from $B_{\mathrm{dR}}(R)$. Then $D_{\mathrm{cris}}(V)^{\nabla = 0} = D_{\mathrm{cris}}^{\nabla}(V)$ and the map $D_{\mathrm{cris}}(V)\otimes_{R_0[\frac{1}{p}]}R[\frac{1}{p}] \rightarrow D_{\mathrm{dR}}(V)$ is compatible with connections. If $V$ is horizontal crystalline, then the isomorphism $D_{\mathrm{cris}}^{\nabla}(V)\otimes_{W(k)[\frac{1}{p}]}R_0[\frac{1}{p}] \rightarrow D_{\mathrm{cris}}(V)$ is compatible with Frobenius. Note that if $V$ is crystalline, then it is horizontal crystalline if and only if the map $D_{\mathrm{cris}}^{\nabla}(V)\otimes_{W(k)[\frac{1}{p}]}R_0[\frac{1}{p}] \rightarrow D_{\mathrm{cris}}(V)$ is an isomorphism, i.e., if and only if $D_{\mathrm{cris}}(V)$ is generated by its parallel elements.

If $V$ is crystalline, we further equip $D_{\mathrm{cris}}(V)\otimes_{R_0[\frac{1}{p}]}R[\frac{1}{p}]$ with the filtration induced by $D_{\mathrm{dR}}(V)$. Then we have $D_{\mathrm{cris}}(V) \in \mathrm{MF}^{\mathrm{pwa}}(R)$ by \cite[Proposition 8.3.4]{brinon-relative}. 

For $D \in \mathrm{MF}(R)$, define 
\[
V_{\mathrm{cris}}(D) \coloneqq (D\otimes_{R_0[\frac{1}{p}]}B_{\mathrm{cris}}(R))^{\nabla = 0, \varphi = 1} \cap \mathrm{Fil}^0((D_R\otimes_{R[\frac{1}{p}]}B_{\mathrm{dR}}(R))^{\nabla = 0})
\]
where $D\otimes_{R_0[\frac{1}{p}]}B_{\mathrm{cris}}(R)$ is equipped with the Frobenius and connection given by the tensor product, and $D_R\otimes_{R[\frac{1}{p}]}B_{\mathrm{dR}}(R)$ is equipped with the filtration and connection given by the tensor product. Then $V_{\mathrm{cris}}(D)$ is a continuous $\mathbf{Q}_p$-representation of $\mathcal{G}_R$. We say $D$ is \emph{admissible} if there exists a crystalline representation $V$ such that $D \cong D_{\mathrm{cris}}(V)$ in $\mathrm{MF}(R)$, and denote by $\mathrm{MF}^{\mathrm{a}}(R)$ the full subcategory of $\mathrm{MF}^{\mathrm{pwa}}(R)$ consisting of admissible modules. Then by \cite[Theorem 8.5.2]{brinon-relative}, $D_{\mathrm{cris}}$ and $V_{\mathrm{cris}}$ are quasi-inverse equivalences of Tannakian categories between the category of crystalline representations and $\mathrm{MF}^{\mathrm{a}}(R)$. It is not known precisely which punctually weakly admissible modules are admissible.

We say $D \in \mathrm{MF}(R)$ is \emph{weakly admissible} if $D$ is punctually weakly admissible and there exists a finite \'{e}tale extension $R_0'$ over $R_0$ such that $D_{\mathrm{cris}}(V)\otimes_{R_0[\frac{1}{p}]}R_0'[\frac{1}{p}]$ is free over $R_0'[\frac{1}{p}]$. For characters, $D_{\mathrm{cris}}$ induces an equivalence between the category of crystalline characters of $\mathcal{G}_R$ and the category of weakly admissible $R_0[\frac{1}{p}]$-modules of rank $1$. However, it is not known whether $D_{\mathrm{cris}}(V)$ is weakly admissible for any crystalline representation $V$.

\subsection{Relative $B$-pairs} \label{sec:2.2}

In \cite{berger-B-pair}, $B$-pairs are studied when the base is a $p$-adic field. There is a natural fully faithful functor from the category of $\mathbf{Q}_p$-representations to the category of $B$-pairs, and the category of $B$-pairs is equivalent to that of $(\varphi, \Gamma)$-modules over the Robba ring. We define the category of $B$-pairs in the relative case and study its relations to $\mathbf{Q}_p$-representations and admissible modules. 

Let $B_e(R) \coloneqq (B_{\mathrm{cris}}^{\nabla}(R))^{\varphi = 1}$. A $B$-\emph{pair} $W = (W_e, W_{\mathrm{dR}}^{\nabla +})$ is given by a finite free $B_e(R)$-module $W_e$ equipped with a semi-linear $\mathcal{G}_R$-action and a finite free $B_{\mathrm{dR}}^{\nabla +}(R)$-module $W_{\mathrm{dR}}^{\nabla +}$ equipped with a semi-linear $\mathcal{G}_R$-action such that 
\[
W_e\otimes_{B_e(R)} B_{\mathrm{dR}}^{\nabla}(R) \cong W_{\mathrm{dR}}^{\nabla +}\otimes_{B_{\mathrm{dR}}^{\nabla +}(R)} B_{\mathrm{dR}}^{\nabla}(R)
\]
as $B_{\mathrm{dR}}^{\nabla}(R)$-modules compatible with $\mathcal{G}_R$-actions. Denote by $B\mathrm{-Pair}(R)$ the category of $B$-pairs whose morphisms are pairs of $B_e(R)$-module and $B_{\mathrm{dR}}^{\nabla +}(R)$-module morphisms compatible with $\mathcal{G}_R$-actions and with isomorphisms over $B_{\mathrm{dR}}^{\nabla}(R)$.

We have a natural functor $W_R$ from the category of $\mathbf{Q}_p$-representations of $\mathcal{G}_R$ to $B\mathrm{-Pair}(R)$ given by $W_R(V) = (V\otimes_{\mathbf{Q}_p}B_e(R), ~V\otimes_{\mathbf{Q}_p}B_{\mathrm{dR}}^{\nabla +}(R))$.  

\begin{prop} \label{prop:2.3}
The functor $W_R$ is fully faithful. Furthermore, if $V$ is a crystalline $\mathcal{G}_R$-representation, then
\[
W_R(V) \cong ((D_{\mathrm{cris}}(V)\otimes_{R_0[\frac{1}{p}]}B_{\mathrm{cris}}(R))^{\nabla = 0, \varphi = 1}, ~\mathrm{Fil}^0(D_{\mathrm{dR}}(V)\otimes_{R[\frac{1}{p}]}B_{\mathrm{dR}}(R))^{\nabla = 0})
\]
as $B$-pairs.
\end{prop}

\begin{proof}
We have $B_e(R) \cap B_{\mathrm{dR}}^{\nabla +}(R) = \mathbf{Q}_p$ by Proposition \ref{prop:2.2}, so $W_R$ is fully faithful. 

If $V$ is crystalline, then the maps
\[
\alpha_{\mathrm{cris}}: D_{\mathrm{cris}}(V)\otimes_{R_0[\frac{1}{p}]} B_{\mathrm{cris}}(R)  \rightarrow V\otimes_{\mathbf{Q}_p}  B_{\mathrm{cris}}(R)
\]
and 
\[
\alpha_{\mathrm{dR}}: D_{\mathrm{dR}}(V)\otimes_{R[\frac{1}{p}]} B_{\mathrm{dR}}(R)  \rightarrow V\otimes_{\mathbf{Q}_p} B_{\mathrm{dR}}(R)
\]
are isomorphisms. Thus,
\[
(V\otimes_{\mathbf{Q}_p} B_{\mathrm{cris}}(R))^{\nabla = 0, \varphi = 1} = V\otimes_{\mathbf{Q}_p}B_e(R) \cong (D_{\mathrm{cris}}(V)\otimes_{R_0[\frac{1}{p}]}B_{\mathrm{cris}}(R))^{\nabla = 0, \varphi = 1}
\]
and
\[
\mathrm{Fil}^0(V\otimes_{\mathbf{Q}_p} B_{\mathrm{dR}}(R))^{\nabla = 0} = V\otimes_{\mathbf{Q}_p}B_{\mathrm{dR}}^{\nabla +}(R) \cong \mathrm{Fil}^0(D_{\mathrm{dR}}(V)\otimes_{R[\frac{1}{p}]}B_{\mathrm{dR}}(R))^{\nabla = 0}.
\]
Furthermore, the diagram connecting $\alpha_{\mathrm{cris}}$ and $\alpha_{\mathrm{dR}}$ induced by $D_{\mathrm{cris}}(V)\otimes_{R_0[\frac{1}{p}]}R[\frac{1}{p}] \cong D_{\mathrm{dR}}(V)$ and the embedding $B_{\mathrm{cris}}(R)\otimes_{R_0[\frac{1}{p}]}R[\frac{1}{p}] \rightarrow B_{\mathrm{dR}}(V)$ is commutative. This proves the second statement. 
\end{proof}

\noindent If we denote by $B\mathrm{-Pair}^{\mathrm{rep}}(R)$ the full subcategory of $B\mathrm{-Pair}(R)$ given by the essential image of $W_R$, then the functor $(W_e, W_{\mathrm{dR}}^{\nabla +}) \mapsto W_e \cap W_{\mathrm{dR}}^{\nabla +}$ from $B\mathrm{-Pair}^{\mathrm{rep}}(R)$ to the category of $\mathbf{Q}_p$-representations is a quasi-inverse to $W_R$ by Proposition \ref{prop:2.2}.  

For any weakly admissible $R_0[\frac{1}{p}]$-module $D$, we denote $W_e(D) \coloneqq (D\otimes_{R_0[\frac{1}{p}]}B_{\mathrm{cris}}(R))^{\nabla = 0, \varphi = 1}$ and $W_{\mathrm{dR}}^{\nabla +}(D) \coloneqq \mathrm{Fil}^0(D_R\otimes_{R[\frac{1}{p}]}B_{\mathrm{dR}}(R))^{\nabla = 0}$.

\section{Horizontal crystalline representations of rank $2$ when $R = \mathcal{O}_K[\![Y]\!]$ and $k = \overline{k}$} \label{sec:3}

In this section, we consider the case when $R_0 = W(k)[\![Y]\!]$ with algebraically closed residue field $k$ and prove Theorem \ref{thm:1.3}. Let $R = R_0\otimes_{W(k)}\mathcal{O}_K$ as above, and choose a uniformizer $\varpi \in \mathcal{O}_K$. Equip $R_0$ with the Frobenius given by $Y \mapsto Y^p$. Note that $R_0$ is isomorphic to the completion of $W(k)\langle X^{\pm 1} \rangle$ with respect to the ideal $(p, X-1)$ via $X \mapsto Y+1$. Thus, by Proposition \ref{prop:2.1}, if we let $u = (Y+1)\otimes 1-1\otimes [\widetilde{Y+1}] \in R\otimes_{W(k)} W(\overline{R}^\flat)$, then 
\[
B_{\mathrm{dR}}^{\nabla +}(R)[\![u]\!] = B_{\mathrm{dR}}^+(R), ~~\mathrm{Fil}^0 B_{\mathrm{dR}}(R) = B_{\mathrm{dR}}^+(R)[\frac{u}{t}].
\]

Let $V$ be a horizontal crystalline $\mathcal{G}_R$-representation of rank $2$ whose Hodge-Tate weights lie in $[0, 1]$. $D_{\mathrm{cris}}^{\nabla}(V)$ is an isocrystal over $W(k)[\frac{1}{p}]$, and we have a $\varphi$-equivariant isomorphism $D_{\mathrm{cris}}^{\nabla}(V)\otimes_{W(k)[\frac{1}{p}]}R_0[\frac{1}{p}] \cong D_{\mathrm{cris}}(V)$. Denote $D = D_{\mathrm{cris}}(V)$ and $D_R^1 = \mathrm{Fil}^1 D_R$. We say $D$ is \'{e}tale (resp. multiplicative) if $D_R^1 = D_R$ (resp. $D_R^1 = 0$). If $D$ is \'{e}tale (resp. multiplicative), then it is induced from an \'{etale} (resp. a multiplicative) filtered $\varphi$-module $D_{\mathrm{cris}}^{\nabla}(V)$. In particular, $V$ arises from a $p$-divisible group over $R$ in both \'{e}tale and multiplicative cases.

Now, assume $D_R^1$ has rank $1$ over $R[\frac{1}{p}]$. Then $t_H(D) = 1$, and both $D_R^1$ and $D_R/D_R^1$ are free over $R[\frac{1}{p}]$ since $R[\frac{1}{p}]$ is a principal ideal domain. Suppose further that $D_{\mathrm{cris}}^{\nabla}(V)$ is reducible as an isocrytal over $W(k)[\frac{1}{p}]$. Since $k = \overline{k}$, we can apply the Dieudonn\'{e}-Manin classification. Note that $D$ is weakly admissible and thus the slopes of all isoclinic subobjects of $D_{\mathrm{cris}}^{\nabla}(V)$ are non-negative, since each isoclinic subobject of $D_{\mathrm{cris}}^{\nabla}(V)$ induces a subobject of $D$. Hence, we can choose a $W(k)[\frac{1}{p}]$-basis $(e_1, e_2)$ of $D_{\mathrm{cris}}^{\nabla}(V)$ such that 
\[
\begin{split}
\varphi(e_1) &= pe_1,\\
\varphi(e_2) & = e_2.	
\end{split}
\]
Then 
\[
W_e(D) = (D\otimes_{R_0[\frac{1}{p}]} B_{\mathrm{cris}}(R))^{\nabla = 0, \varphi = 1} = \frac{1}{t}e_1\cdot B_e \oplus e_2\cdot B_e.
\]
On the other hand, since $D_R^1$ and $D_R/D_R^1$ are free of rank $1$ over $R[\frac{1}{p}]$, we have 
\[
D_R^1 = (f(Y)e_1+g(Y)e_2)\cdot R[\frac{1}{p}]
\] 
for some $f(Y), g(Y) \in \mathcal{O}_K[\![Y]\!]$ such that either $\varpi \nmid f(Y)$ or $\varpi \nmid g(Y)$ in $\mathcal{O}_K[\![Y]\!]$ and that there exist $h(Y), r(Y) \in \mathcal{O}_K[\![Y]\!]$ with $(f(Y)r(Y)-g(Y)h(Y))$ being a unit in $R[\frac{1}{p}]$. Then, 
\[
D_R \cong (f(Y)e_1+g(Y)e_2)\cdot R[\frac{1}{p}] \oplus (h(Y)e_1+r(Y)e_2)\cdot R[\frac{1}{p}]
\]
as $R[\frac{1}{p}]$-modules, and
\[
\mathrm{Fil}^0 (D_R\otimes_{R[\frac{1}{p}]} B_{\mathrm{dR}}(R)) = \frac{f(Y)e_1+g(Y)e_2}{t}\cdot \mathrm{Fil}^0 B_{\mathrm{dR}}(R)\oplus (h(Y)e_1+r(Y)e_2)\cdot \mathrm{Fil}^0 B_{\mathrm{dR}}(R).
\]
Denote $c = [\widetilde{Y+1}]-1$, so that $Y = u+c$. We can write $f(Y) = f(u+c) = f(c)+uf_1(u)$ with $f(c) \in B_{\mathrm{dR}}^{\nabla +}(R)$ and $f_1(u) \in B_{\mathrm{dR}}^+(R) = B_{\mathrm{dR}}^{\nabla +}[\![u]\!]$, and similarly for $g(Y), h(Y), r(Y)$. For any $a(u), b(u) \in B_{\mathrm{dR}}^+(R)$, we have
\[
\begin{split}
&\frac{f(Y)e_1+g(Y)e_2}{t}(1+ua(u))+(h(Y)e_1+r(Y)e_2)\frac{u}{t}b(u) = \\
&\frac{f(c)e_1+g(c)e_2}{t}+\frac{u}{t}((f(Y)a(u)+h(Y)b(u)+f_1(u))e_1+(g(Y)a(u)+r(Y)b(u)+g_1(u))e_2).
\end{split}
\]
The system of equations
\[
\begin{split}
f(Y)a(u)+h(Y)b(u) &= -f_1(u),\\
g(Y)a(u)+r(Y)b(u) &= -g_1(u)
\end{split}
\]
has a unique solution $a(u), b(u) \in B_{\mathrm{dR}}^+(R)$, since $f(Y)r(Y)-g(Y)h(Y)$ is a unit in $R[\frac{1}{p}]$. Thus, 
\[
\frac{f(c)e_1+g(c)e_2}{t} \in W_{\mathrm{dR}}^{\nabla +}(D) = \mathrm{Fil}^0 (D\otimes_R B_{\mathrm{dR}}(R))^{\nabla = 0}.
\]
Similarly, for any $a(u), b(u) \in B_{\mathrm{dR}}^+(R)$, 
\[
\begin{split}
&\frac{f(Y)e_1+g(Y)e_2}{t}\cdot tua(u)+(h(Y)e_1+r(Y)e_2)(1+ub(u)) =\\
& h(c)e_1+r(c)e_2+u((f(Y)a(u)+h(Y)b(u)+h_1(u))e_1+(g(Y)a(u)+r(Y)b(u)+r_1(u))e_2)	,
\end{split}
\]
and the system
\[
\begin{split}
f(Y)a(u)+h(Y)b(u) &= -h_1(u),\\
g(Y)a(u)+r(Y)b(u) &= -r_1(u)
\end{split}
\] 
has a unique solution $a(u), b(u) \in B_{\mathrm{dR}}^+(R)$. Thus, $h(c)e_1+r(c)e_2 \in W_{\mathrm{dR}}^{\nabla +}(D)$. We then have
\[
W_{\mathrm{dR}}^{\nabla +}(D) = \frac{f(c)e_1+g(c)e_2}{t}\cdot B_{\mathrm{dR}}^{\nabla +}(R) \oplus (h(c)e_1+r(c)e_2)\cdot B_{\mathrm{dR}}^{\nabla +}(R).
\]
Note that $W_e(D)\otimes_{B_e(R)}B_{\mathrm{dR}}^{\nabla}(R) = W_{\mathrm{dR}}^{\nabla +}(D)\otimes_{B_{\mathrm{dR}}^{\nabla +}(R)} B_{\mathrm{dR}}^{\nabla}(R) = e_1\cdot B_{\mathrm{dR}}^{\nabla}(R) \oplus e_2\cdot B_{\mathrm{dR}}^{\nabla}(R)$, since $f(c)r(c)-g(c)h(c)$ is a unit in $B_{\mathrm{dR}}^{\nabla +}(R)$. In particular, $(W_e(D), W_{\mathrm{dR}}^{\nabla +}(D))$ is a $B$-pair. 

The intersection $W_e(D) \cap W_{\mathrm{dR}}^{\nabla +}(D)$ is given by the set of solutions $(x, y, s, z)$ with $x, y \in B_e(R), ~s, z \in B_{\mathrm{dR}}^{\nabla +}(R)$ satisfying
\[
\begin{split}
\frac{x}{t} &= \frac{f(c)}{t}s+h(c)z,\\	
y &= \frac{g(c)}{t}s+r(c)z. 
\end{split}
\]
Then $x = f(c)s+th(c)z \in B_e(R) \cap B_{\mathrm{dR}}^{\nabla +}(R) = \mathbf{Q}_p$ by Proposition \ref{prop:2.2}, and $\displaystyle y = \frac{y_1}{t}$ with $y_1 \in U_1$.

\begin{prop} \label{prop:3.1}
$f(Y)$ is a unit in $R[\frac{1}{p}]$, and for any such $D$, $V_{\mathrm{cris}}(D) = W_e(D) \cap W_{\mathrm{dR}}^{\nabla +}(R)$ has rank $2$ over $\mathbf{Q}_p$.	
\end{prop}

\begin{proof}
Suppose $f(Y)$ is not a unit in $R[\frac{1}{p}]$. Since $\theta(c) = Y$, by applying $\theta$ to the equations
\[
\begin{split}
	x &= f(c)s+th(c)z,\\
	y_1 &= g(c)s+tr(c)z,
\end{split}
\]
we obtain 
\[
\begin{split}
	x &= f(Y)\theta(s),\\
	\theta(y_1) &= g(Y)\theta(s).
\end{split}
\]
Since $x \in \mathbf{Q}_p$ and $\theta(s) \in \widehat{\overline{R}}[\frac{1}{p}]$, we have $x = 0 = \theta(s)$. Then $\theta(y_1) = 0$, and by Proposition \ref{prop:2.2}, the $\mathbf{Q}_p$-module $W_e(D) \cap W_{\mathrm{dR}}^{\nabla +}(D)$ has rank $1$. This contradicts to $D$ being admissible. 

Since $f(Y)$ is a unit in $R[\frac{1}{p}]$, we can choose $h(Y) = 0$ and $r(Y) = 1$. Then $x \in \mathbf{Q}_p$ as above and $s = f(c)^{-1}x \in B_{\mathrm{dR}}^{\nabla +}(R)$. We have $\theta(y_1) = g(Y)\theta(s) = g(Y)f(Y)^{-1}x$ and $\displaystyle z = \frac{y_1-g(c)s}{t} \in B_{\mathrm{dR}}^{\nabla +}(R)$. Hence, $V_{\mathrm{cris}}(D)$ has rank $2$ over $\mathbf{Q}_p$ by Proposition \ref{prop:2.2}. 
\end{proof}

By Proposition \ref{prop:3.1}, we can write $D_R^1 = (e_1+g(Y)e_2)\cdot R[\frac{1}{p}]$ for some $g(Y) \in R[\frac{1}{p}]$. By replacing $e_2$ by $\frac{1}{p^n}e_2$ for some non-negative integer $n$ if necessary, we can further assume $g(Y) = pg_1(Y)$ for some $g_1(Y) \in R = \mathcal{O}_K[\![Y]\!]$. We now show such $D$ arises from a $p$-divisible group over $R$ by constructing the associated relative Breuil module. Let $E(u)$ be the Eisenstein polynomial for $\varpi$ over $R_0$, and let $\mathfrak{S} = R_0[\![u]\!]$ equipped with the Frobenius extending that on $R_0$ by $u \mapsto u^p$. Let $S$ be the $p$-adic completion of the divided power envelope of $\mathfrak{S}$ with respect to the ideal $(E(u))$. Explicitly, the elements of $S$ can be described as 
\[
S = \{\sum_{n \geq 0} a_n\frac{u^n}{\lfloor n/e\rfloor!}~|~ a_n \in R_0, ~a_n \rightarrow 0 ~p\mbox{-adically}\}
\]
where $e$ is the degree of $E(u)$. Note that $S/(E(u)) \cong R$, and the Frobenius on $\mathfrak{S}$ extends uniquely to $S$. Let $\mathrm{Fil}^1 S \subset S$ be the $p$-adically completed ideal generated by the divided powers $\displaystyle \frac{E(u)^n}{n!}, ~n \geq 1$. Since $\displaystyle\frac{\varphi(E(u))}{p}$ is a unit in $S$, we have $\varphi(\mathrm{Fil}^1 S) = pS$ as $S$-modules. Denote by $d_S^u: S \rightarrow S\otimes_{R_0} \widehat{\Omega}_{R_0}$ the connection given by 
\[
d_S^u(\sum_{n \geq 0} a_n\frac{u^n}{\lfloor n/e\rfloor!}) = \sum_{n \geq 0} \frac{u^n}{\lfloor n/e\rfloor!}d_{R_0}(a_n),
\]
where $d_{R_0}: R_0 \rightarrow R_0\otimes_{R_0}\widehat{\Omega}_{R_0}$ is the universal connection and $a_n \in R_0$ such that $a_n \rightarrow 0$ in the $p$-adic topology.

Let $\mathfrak{g}_1 \in S$ be any preimage of $g_1(Y)$ for the map $S \twoheadrightarrow S/\mathrm{Fil}^1 S \cong R$. Let $\mathcal{M} = e_1\cdot S \oplus e_2\cdot S$, equipped with the filtration 
\[
\mathrm{Fil}^1 \mathcal{M} = \mathrm{Fil}^1 S\cdot \mathcal{M}+(e_1+p\mathfrak{g}_1e_2)\cdot S.
\] 
Note that $\mathcal{M}/\mathrm{Fil}^1 \mathcal{M} \cong D_R/D_R^1$ as $R$-modules. Equip $\mathcal{M}$ with the Frobenius given by $\varphi(e_1) = pe_1, ~\varphi(e_2) = e_2$ as above. Then $\varphi(\mathrm{Fil}^1 \mathcal{M}) = p\mathcal{M}$ as $S$-modules. Let $\nabla: \mathcal{M} \rightarrow \mathcal{M}\otimes_{R_0} \widehat{\Omega}_{R_0}$ be the connection over $d_S^u$ given by $\nabla(e_1) = \nabla(e_2) = 0$. Then $\nabla$ is a topologically quasi-nilpotent integrable connection such that the Frobenius is horizontal. Hence, $\mathcal{M}$ gives a Breuil module over $S$ (as defined in \cite[Section 3]{kim-groupscheme-relative}), and by \cite[Theorem 3.5]{kim-groupscheme-relative}, there exists a $p$-divisible group $G_R$ over $R$ such that $\mathcal{M}^*(G_R) \cong \mathcal{M}$ as Breuil modules (cf. \cite{kim-groupscheme-relative} for the definition of the functor $\mathcal{M}^*(\cdot)$ from the category of $p$-divisible groups over $R$ to the category of Breuil modules over $S$). 

For integers $n \geq 0$, we choose compatibly $\varpi_n \in \overline{K}$ such that $\varpi_0 = \varpi$ and $\varpi_{n+1}^p = \varpi_n$, and let $R_{\infty}$ be the $p$-adic completion of $\bigcup_{n \geq 0} R(\varpi_n)$. Then $R_{\infty} \subset \overline{R}$, and let $\mathcal{G}_{R_{\infty}}\coloneqq \mathrm{Gal}(\overline{R}[\frac{1}{p}]/R_{\infty}[\frac{1}{p}])$ be the corresponding sub-Galois group of $\mathcal{G}_R$. Let $[\underline{\varpi}] \in W(\overline{R}^\flat)$ be the Teichm\"uller lift of $\underline{\varpi} = (\varpi_n) \in \overline{R}^\flat$. The $R_0$-algebra map $\mathfrak{S} \rightarrow R_0\otimes_{W(k)}W(\overline{R}^\flat)$ given by $u \mapsto [\underline{\varpi}]$ extends uniquely to $S \rightarrow A_{\mathrm{cris}}(R)$ compatibly with $\mathcal{G}_{R_\infty}$-actions, Frobenius and connection. Let $\mathrm{Fil}^1 A_{\mathrm{cris}}(R) \subset A_{\mathrm{cris}}(R)$ be the $p$-adically completed ideal generated by the divided powers of $\mathrm{ker}(\theta_{R_0})$. Then $S \rightarrow A_{\mathrm{cris}}(R)$ is also compatible with filtration. Define
\[
T(\mathcal{M}) \coloneqq \mathrm{Hom}_{S, \mathrm{Fil}^1, \varphi, \nabla}(\mathcal{M}, A_{\mathrm{cris}}(R)).
\]    
By \cite[Corollary 5.4.2]{kim-groupscheme-relative}, we have a natural isomorphism $T_p(G_R) \cong T(\mathcal{M})$ as $\mathcal{G}_{R_\infty}$-representations. 

To study $\mathcal{G}_R$-actions, let $N: S \rightarrow S$ be a derivation given by $N \coloneqq -u\frac{\partial}{\partial u}$. Let $N_{\mathcal{M}}: \mathcal{M} \rightarrow \mathcal{M}$ be the derivation over $N$ given by $N_{\mathcal{M}}(e_1) = N_{\mathcal{M}}(e_2) = 0$. For each integer $n \geq 0$, define a cocycle $\epsilon^{(n)}: \mathcal{G}_R \rightarrow \widehat{\overline{R}}^{\times}$ by 
\[
\epsilon^{(n)}(g) = g\cdot \varpi_n/\varpi_n
\]
for $g \in \mathcal{G}_R$. Let $\epsilon(g) = (\epsilon^{(n)}(g))_{n \geq 0} \in \overline{R}^\flat$, and $t(g) \coloneqq \log{[\epsilon(g)]} \in A_{\mathrm{cris}}^{\nabla}(R)$. Note that for any $g \in \mathcal{G}_R$, $t(g)$ is a $\mathbf{Z}_p$-multiple of $t$, and $t(g) = 0$ if and only if $g \in \mathcal{G}_{R_{\infty}}$. Define the $\mathcal{G}_R$-action on $\mathcal{M}\otimes_{S}A_{\mathrm{cris}}(R)$ by 
\[
g\cdot (x\otimes a) \coloneqq g(a)\sum_{i = 0}^{\infty} N_{\mathcal{M}}^i(x)\otimes \frac{t(g)^i}{i!}.
\]
Then by \cite[Section 5.5]{kim-groupscheme-relative}, this gives a well-defined $\mathcal{G}_R$-action on $\mathcal{M}\otimes_{S}A_{\mathrm{cris}}(R)$ which recovers the natural $\mathcal{G}_{R_{\infty}}$-action and is compatible with Frobenius and filtration. This induces a $\mathcal{G}_R$-action on $T(\mathcal{M}) \cong \mathrm{Hom}_{A_{\mathrm{cris}}(R), \mathrm{Fil}^1, \varphi, \nabla}(\mathcal{M}\otimes_S A_{\mathrm{cris}}(R), ~A_{\mathrm{cris}}(R))$, and the natural isomorphism $T_p(G_R) \cong T(\mathcal{M})$ is $\mathcal{G}_R$-equivariant.

Consider the map of $R_0[\frac{1}{p}]$-modules $f: D \rightarrow \mathcal{M}[\frac{1}{p}]$ given by $e_i \mapsto e_i$. Then $f$ is compatible with Frobenius and connection. Thus, $f$ induces the map 
\[
\tilde{f}: T(\mathcal{M})[\frac{1}{p}] \rightarrow \mathrm{Hom}_{R_0[\frac{1}{p}], \varphi, \nabla}(D, B_{\mathrm{cris}}(R))
\]
which is compatible with $\mathcal{G}_{R_{\infty}}$-actions. Note that $\tilde{f}$ is injective. Furthermore, since $f(D)$ lies in the kernel of $\mathcal{N}_{\mathcal{M}}$, $\tilde{f}$ is compatible with $\mathcal{G}_R$-actions. On the other hand, since $\varpi - [\underline{\varpi}] \in R\otimes_{W(k)}W(\overline{R}^\flat)$ lies in $\mathrm{ker}(\theta_R)$ and $\nabla(\varpi - [\underline{\varpi}]) = 0$, the image of $\tilde{f}$ lies in 
\[
\mathrm{Hom}_{R_0[\frac{1}{p}], \varphi, \nabla}(D, B_{\mathrm{cris}}(R)[\frac{1}{p}]) \cap \mathrm{Hom}_{R[\frac{1}{p}], \mathrm{Fil}, \nabla}(D_R, B_{\mathrm{dR}}(R)) \cong V_{\mathrm{cris}}(D)^{\vee},
\] 
where $V_{\mathrm{cris}}(D)^{\vee}$ denotes the dual representation of $V_{\mathrm{cris}}(D)$. So $\tilde{f}$ induces an injective map $\tilde{f}: T(\mathcal{M})[\frac{1}{p}] \hookrightarrow V_{\mathrm{cris}}(D)^{\vee}$ of $\mathbf{Q}_p$-vector spaces, and $V_{\mathrm{cris}}(D)^{\vee}$ is rank $2$ over $\mathbf{Q}_p$ by Proposition \ref{prop:3.1}. Thus, it is an isomorphism, and we have an isomorphism of $\mathcal{G}_R$-representations
\[
V_{\mathrm{cris}}(D) \cong (T_p(G_R)\otimes_{\mathbf{Z}_p}\mathbf{Q}_p)^{\vee}.
\]

Using Proposition \ref{prop:3.1}, we can also construct $B$-pairs which are induced from weakly admissible $R_0[\frac{1}{p}]$-modules but do not arise from $\mathbf{Q}_p$-representations. For example, consider the case $K = W(k)[\frac{1}{p}]$ (so that $R = R_0$ is unrafmied), and let $D = e_1\cdot R_0[\frac{1}{p}]\oplus e_2\cdot R_0[\frac{1}{p}]$ equipped with the filtration $\mathrm{Fil}^0 D_R = D_R, ~\mathrm{Fil}^1 D_R = ((Y+p)e_1+e_2)\cdot R[\frac{1}{p}]$ and $\mathrm{Fil}^2 D_R = 0$. Equip $D$ with Frobenius endomorphism given by
\[
\begin{split}
	\varphi(e_1) &= pe_1,\\
	\varphi(e_2) &= e_2,
\end{split}
\]
and equip $D$ with the connection given by $\nabla(e_1) = \nabla(e_2) = 0$. We have $t_H(D) = 1$ and $t_N(D, \mathfrak{p}) = 1$ for any $\mathfrak{p} \in \mathrm{Spec} R_0/pR_0$. Consider $\varphi$-equivariant base change maps $b_0: R_0 = W(k)[\![Y]\!] \rightarrow W(k)$ given by $Y \mapsto 0$ and $b_g: R_0 \rightarrow R_{0, g}$ where $R_{0, g}$ is the $p$-adic completion of $\varinjlim_{\varphi}R_{0, (p)}$. Note that by the universal property of $p$-adic Witt vectors, we have a $\varphi$-equivariant isomorphism $R_{0, g} \cong W(k_g)$ where $k_g$ is the perfect closure $\varinjlim_{\varphi}\mathrm{Frac}(R_0/pR_0)$ of $\mathrm{Frac}(R_0/pR_0)$. To check $D$ is weakly admissible, it suffices to show that the induced filtered $\varphi$-modules $D_0 \coloneqq D\otimes_{R, b_0} W(k)$ and $D_g \coloneqq D\otimes_{R, b_g} W(k_g)$ are weakly admissible. We have $D_0 = \overline{e_1}\cdot W(k)[\frac{1}{p}] \oplus \overline{e_2}\cdot W(k)[\frac{1}{p}]$ with $\mathrm{Fil}^1 D_0 = (p\overline{e_1}+\overline{e_2})\cdot W(k)[\frac{1}{p}]$. It admits a strongly divisible $W(k)$-lattice $M_0 = p\overline{e_1}\cdot W(k)\oplus \frac{\overline{e_2}}{p}\cdot W(k)$ with $\mathrm{Fil}^1 M_0 = (p\overline{e_1}+\overline{e_2})\cdot W(k)$, so $D_0$ is weakly admissible. On the other hand, $D_g = e_1\cdot W(k_g)[\frac{1}{p}] \oplus e_2\cdot W(k_g)[\frac{1}{p}]$ with $\mathrm{Fil}^1 D_g = ((Y+p)e_1+e_2)\cdot W(k_g)[\frac{1}{p}]$. It admits a strongly divisible $W(k_g)$-lattice $M_g = e_1\cdot W(k_g)\oplus \frac{e_2}{p}\cdot W(k_g)$ with $\mathrm{Fil}^1 M_g = ((Y+p)e_1+e_2)\cdot W(k_g)$, so $D_g$ is weakly admissible. Hence, $D$ is a weakly admissible $R_0[\frac{1}{p}]$-module. However, above computations and Proposition \ref{prop:3.1} show that $(W_e(D), W_{\mathrm{dR}}^{\nabla +}(D))$ is a $B$-pair which does not arise from a $\mathbf{Q}_p$-representation, since $(Y+p)$ is not a unit in $R[\frac{1}{p}]$. Thus, the relative case is different from the case when the base ring is a $p$-adic field where every $B$-pair semi-stable of slope $0$ arises from a $\mathbf{Q}_p$-representation. In particular, this answers negatively the question raised in \cite[Section 8]{brinon-relative} whether weakly admissible implies admissible in the relative case.

We summarize above results in the following theorem.

\begin{thm} \label{thm:3.2}
Let $R = \mathcal{O}_K[\![Y]\!]$ whose residue field $k$ is algebraically closed. Let $V$ be a horizontal crystalline $\mathcal{G}_R$-representation of rank $2$ over $\mathbf{Q}_p$ with Hodge-Tate weights in $[0, 1]$ such that its associated isocrystal is reducible. Then $V$ arises from a $p$-divisible group over $R$. Moreover, there exists a $B$-pair which arises from a weakly admissible $R_0[\frac{1}{p}]$-module but does not arise from a $\mathbf{Q}_p$-representation.
\end{thm}

\begin{rem} \label{rem:3.3}
When $R_0$ is a general relative base ring over $W(k)\langle X^{\pm 1}\rangle$ with Krull dimension $2$ and $R = R_0\otimes_{W(k)}\mathcal{O}_K$ with ramification index $e < p-1$, then it is proved in \cite{liu-moon-rel-cryst} using a completely different method that every crystalline $\mathcal{G}_R$-representation with Hodge-Tate weights in $[0, 1]$ arises from a $p$-divisible group over $R$. The argument in \cite{liu-moon-rel-cryst} crucially relies on the assumption that $e$ is small, even for the case $R = \mathcal{O}_K[\![Y]\!]$.	
\end{rem}

\bibliographystyle{amsalpha}
\bibliography{library}

\providecommand{\bysame}{\leavevmode\hbox to3em{\hrulefill}\thinspace}
\providecommand{\MR}{\relax\ifhmode\unskip\space\fi MR }
\providecommand{\MRhref}[2]{%
  \href{http://www.ams.org/mathscinet-getitem?mr=#1}{#2}
}
\providecommand{\href}[2]{#2}
\begin{thebibliography}{Kim15}

\bibitem[Ber08]{berger-B-pair}
Laurent Berger, \emph{Construction de $(\varphi, \gamma)$-modules:
  repr\'{e}sentations $p$-adiques et $b$-paires}, {A}lgebra {N}umber {T}heory
  \textbf{2} (2008), 91--120.

\bibitem[Bri08]{brinon-relative}
Olivier Brinon, \emph{Repr\'{e}sentations {$p$}-adiques cristallines et de de
  rham dans le cas relatif}, M\'{e}m. {S}oc. {M}ath. {F}r. \textbf{112} (2008).

\bibitem[CF00]{colmez-fontaine}
Pierre Colmez and Jean-Marc Fontaine, \emph{Construction des
  repr\'{e}sentations {$p$}-adiques semi-stables}, Invent. {M}ath. \textbf{140}
  (2000), 1--43.

\bibitem[Kim15]{kim-groupscheme-relative}
Wansu Kim, \emph{The relative {B}reuil-{K}isin classification of
  {$p$}-divisible groups and finite flat group schemes}, Int. {M}ath. {R}es.
  {N}ot. (2015), 8152--8232.

\bibitem[Kis06]{kisin-crystalline}
Mark Kisin, \emph{Crystalline representations and {$F$}-crystals}, Algebraic
  geometry and number theory (Boston), Progr. {M}ath., vol. 253,
  Birkh\"{a}user, 2006, pp.~459--496.

\bibitem[KL15]{kedlaya-liu-relative-padichodge}
Kiran Kedlaya and Ruochuan Liu, \emph{Relative {$p$}-adic hodge theory:
  foundations}, Ast\'{e}risque \textbf{371} (2015).

\bibitem[LM18]{liu-moon-rel-cryst}
Tong Liu and Yong~Suk Moon, \emph{$p$-divisible groups and relative crystalline
  representations when $e < p-1$}, preprint, 2018.

\bibitem[Sch13]{scholze-p-adic-hodge}
Peter Scholze, \emph{$p$-adic hodge theory for rigid-analytic varieties}, Forum
  {M}ath. {P}i \textbf{1} (2013).

\end{thebibliography}
	
\end{document}